\documentclass[smallextended]{svjour3}       % onecolumn (second format)

\smartqed
\usepackage{lineno,hyperref}
\usepackage{graphicx}
\usepackage{amsmath,amssymb}
\usepackage{subcaption}
\def\R{\mathbb{R}}

\def\Z{\mathbb{Z}}

\def\w{\mathbf{w}}

\def\y{\mathbf{y}}

\def\1{\mathbf{1}}
\def\0{\mathbf{0}}

\def\optlimits{\nolimits}

\usepackage{array, tabu}
\usepackage{amssymb}
\usepackage{amsmath}

\usepackage{dsfont} % For 1, R etc
\usepackage{algorithm}
\usepackage{url}
\usepackage{graphicx, multicol, footmisc}
\usepackage{mathrsfs}
\usepackage{times}
\usepackage{epsfig}
{\vfill\null\thispagestyle{plain}}

\usepackage{algpseudocode}
\algdef{SE}[DOWHILE]{Do}{doWhile}{\algorithmicdo}[1]{\algorithmicwhile\ #1}

\usepackage[T1]{fontenc} 
\usepackage[utf8]{inputenc}
\usepackage{textcomp}
\usepackage[greek,french,english]{babel}
\usepackage{float}
\usepackage{dsfont}
\usepackage{array, tabu}
\usepackage{caption}

\usepackage{tikz}
\usetikzlibrary{decorations.pathreplacing}
\usepackage{bm}
\usepackage{enumerate}
\makeindex 

\usepackage{xparse}

\makeatletter
\NewDocumentCommand{\raisedminus}{m}{%
  \raisebox{0.2em}{$\m@th#1{-}$}%
}

\usepackage{hyperref}

\begin{document}

    \title{Learning Discontinuous Piecewise Affine Fitting Functions using Mixed Integer Programming for Segmentation and Denoising
}
%
%\titlerunning{Abbreviated paper title}
% If the paper title is too long for the running head, you can set
% an abbreviated paper title here
%
\author{Ruobing Shen
%\thanks{This work was done when Ruobing Shen was doing his Ph.D. in Universität Heidelberg.}
%\orcidID{0000-1111-2222-3333}
\and
Bo Tang
\and
Leo Liberti
\and
Claudia D'Ambrosio
\and
Stéphane Canu
} 
%\and
%Third Author\inst{3}\orcidID{2222--3333-4444-5555}}
%
\authorrunning{R. Shen et al.}
% First names are abbreviated in the running head.
% If there are more than two authors, 'et al.' is used.
%
\institute{Ruobing Shen\at
Institue of Computer Science, Universität Heidelberg, 69120 Heidelberg, Germany\\
%
%Springer Heidelberg, Tiergartenstr. 17, 69121 Heidelberg, Germany
\email{ruobing.shen@informatik.uni-heidelberg.de}
\and
Bo Tang \at
College of Science, Northeastern University, Boston 02115, USA
\and
Leo Liberti, Claudia D'Ambrosio \at
LIX CNRS, École Polytechnique, Institut Polytechnique de Paris, 91128 Palaiseau, France
\and
Stéphane Canu\at
 INSA de Rouen, Normandie Université, 76801 Saint Etienne du Rouvray, France
}
%\\
%\url{http://www.springer.com/gp/computer-science/lncs} \and
%ABC Institute, Rupert-Karls-University Heidelberg, Heidelberg, Germany\\
%\email{\{abc,lncs\}@uni-heidelberg.de}}
%
\date{Received: date / Accepted: date}

\maketitle              % typeset the header of the contribution
\begin{abstract}
Piecewise affine functions are widely used to approximate nonlinear and discontinuous functions. However, most, if not all existing models only deal with fitting continuous functions. 
In this paper, We investigate the problem of fitting a discontinuous piecewise affine function to given data that lie in an orthogonal grid, where no restriction on the partition is enforced (i.e., its geometric shape can be nonconvex). This is useful for segmentation and denoising when data correspond to images.
We propose a novel Mixed Integer Program (MIP) formulation for the piecewise affine fitting problem, where binary variables determines the location of break-points. To obtain consistent partitions (i.e.~image segmentation), we include multi-cut constraints in the formulation. Since the resulting problem is $\mathcal{NP}$-hard, two techniques are introduced to improve the computation. One is to add facet-defining inequalities to the formulation and the other to provide initial integer solutions using a special heuristic algorithm. 
We conduct extensive experiments by some synthetic images as well as real depth images, and the results demonstrate the feasibility of our model.
\end{abstract}
%
%
%
%------------------------------------------
\section{Introduction}
Let $D\subseteq \R^d$ denote the signal domain, and $y: D\rightarrow \R$ denote intensity values of the given signals, possibly with some noise. In this paper, we seek to find (or approximate) a discontinuous piecewise affine function that best fits the data $y$ over $D$, where $D$ is restricted to an orthogonal grid. Although it can be generalized to higher dimensions, we are mostly interested in the scenario when $d=2$, where $D$ is a square grid and $y$ corresponds to natural or depth images.

In statistics, \emph{affine (linear) regression} or \emph{affine fitting} is a widely used approach to model the relationship between the data $y$ and the independent variables~$z\in D$, which are, in our case, the coordinates of $y$. In the parametric model, the relationship is modeled using affine functions. The unknown affine parameters $\beta$ (i.e., slopes and intercepts) are estimated from the given data according to some standard objective functions, such as the well-known mean square error (MSE)~\cite{linearmodel}.

Non-parametric models, on the other hand, assume  that  the  data  distribution  cannot  be  defined  in terms  of  such  a  finite  set  of  parameters $\beta$. Typically, the model grows in size according to the complexity of the data. For instance, one can introduce some fitting variables $w$ to model the data $y$; some assumptions are then made about the connections among these variables.

%Lysaker and Tai~\cite{Lysaker2006} provide two second-order regularizers, and we  will adopt a variant of the discrete version of the  $R_2$ in this paper.
%$$
%R_2(w) = \int_{D} \left(\lvert w_{z_1z_1} \rvert + \lvert  w_{z_2z_2} \rvert \right)dz.
%$$
We call a function $f$ {\it (possibly discontinuous) piecewise affine} over $D$ if there is a partition of $D$ into disjoint subsets $D_1, \ldots, D_k$ such that $f$ is affine when restricted to each $D_i$ (we denote by $f^i$ the function $f$ restricted to $D_i$). Let ${\cal D}$ be the set of all partitions of $D$, and ${\cal F}$ be the set of all piecewise affine functions over $D$, then any choice of $f \in {\cal F}$ defines ${\cal D}'\in {\cal D}$. Moreover, if the partition ${\cal D}'$ is known,  the corresponding $f \in {\cal F} $ can be easily identified by computing the  affine parameters $\beta$ within each region $D_i$ under some objectives (e.g.~MSE).

The problem of piecewise affine fitting has been studied for decades. Numerous clustering based algorithms~\cite{8657940,FerrariTrecate2002ANL,10.1007} are designed for different variants of the problem, but only suffice to find local optimal solutions.
Exact formulations of the problem via MIP are also proposed, but often with restrictions. Examples include the continuous piecewise linear fitting models~\cite{piecewise}, where the domain partition is in a sense pre-defined,
and the fitting function $f$ is restricted to be continuous over $D$.
A general $n$-dimensional piecewise linear fitting problem has been studied in~\cite{Discrete}, and formulated as a parametric model using MIP. But the assumption that the segments are linearly separable does not hold in many practical applications.

%Although focused on $1D$ and $2$D data, our model can be easily generalized to higher dimensions, and we have no restriction on $f$.

In this paper, we will focus on the non-parametric model that finds (or approximates) a possibly discontinuous piecewise affine function $f \in {\cal F}$ to fit the data $y$, where the affine regions are unknown and the affine parameters within each $D^i$ are not explicitly computed. Our problem can be mathematically represented as follows:
\begin{alignat}{2}
\min_{{\cal D'}\in {\cal D}, w} \;\; &\sum_{D_i\subset {\cal D'}} \left(\int_{D_i} |(f^i(z)-y(z)| dz + \lambda Per(D_i)\right)\label{potts}\\
&f(z) = w(z) \tag{\ref{potts}a} \label{potts1},
\end{alignat}
where $Per(D_i)$ denotes the perimeter of $D_i$ and $\lambda$ is a regularization parameter.
The first term measures the quality of data fitting, and the second regularization term is used to balance the former with the number and the boundary length of segments (affine regions), to prevent over-fitting. Note that an absolute fitting term is adopted here to enable a Mixed Integer Linear Programming (MILP) formulation of the model. Being a non-parametric model, further constraints on the fitting variables $w$ will  be defined to model the linearity within $D_i$, in Section~\ref{sec:1dproblem} and \ref{sec:2dproblem}.

Figure~\ref{3Dview} shows a synthetic image with noise that has linear trends and its 3D view, where the horizontal axes ($z_1$ and $z_2$) represent the coordinates of the image pixels. Upon finding a piecewise affine function $f$, we also have a segmentation (to be introduced in next section) of the image into background and four segments, and a denoised image (with fitted value $w$) as a by-product.

\begin{figure}[t]
\center
\includegraphics[width=0.4\linewidth]{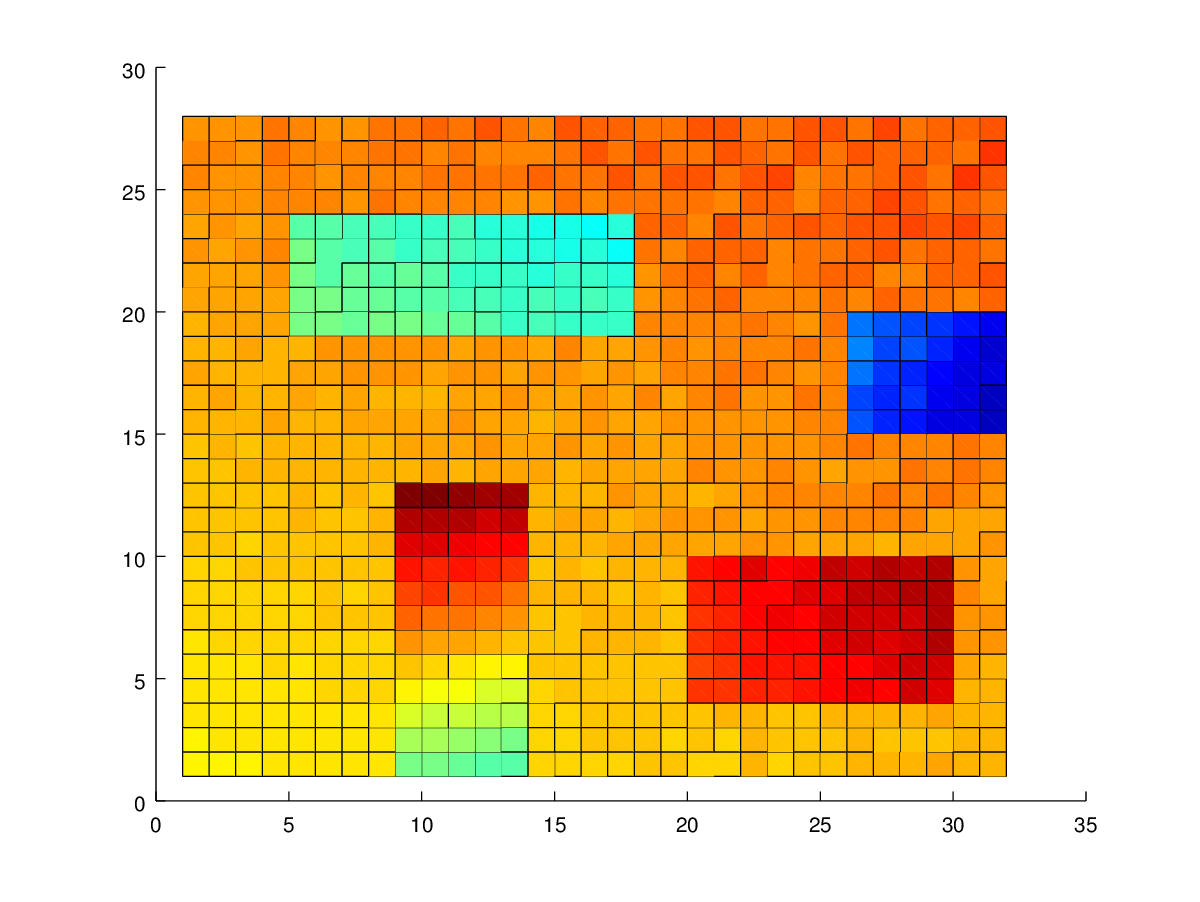}
\includegraphics[width=0.4\linewidth]{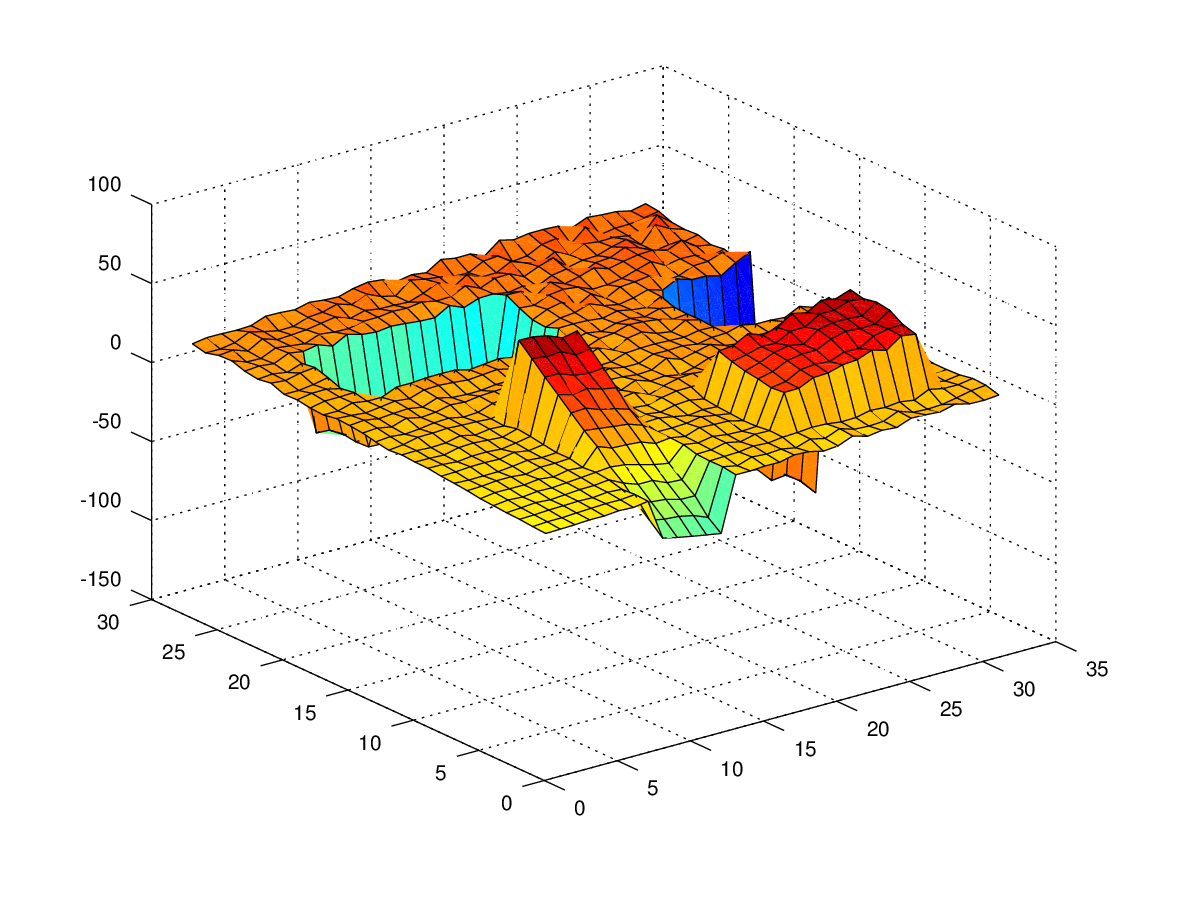}
\caption{A synthetic $2$D image with noise that has linear trend and its $3$D view.}
\label{3Dview}
\end{figure}

%---------------------
\subsection{Related work}

For denoising (fitting) $2$D images, the total variation (TV) model~\cite{rof1992} is widely used
\begin{alignat}{2}
\min_{w} \;\; \int_{D} (w(z)-y(z))^2 dz + \lambda \text{TV}(w),
\label{ROF}
\end{alignat}
where the first part is the squared data fitting term ($w$ the fitting function) and the second part is the regularization term. The TV regularizers can be either isotropic or anisotropic. The latter can be mathematically described as follows:
\[
\text{TV}_{ani}(w) = \int_{D} \left(\lvert \partial_{z_1} w \rvert + \lvert  \partial_{z_2} w \rvert \right)dz,
\]
where $\partial_{z_i}$ represents the partial derivative of $w$ with respect to $z_i$, for $i=1,2$.
Lysaker and Tai~\cite{Lysaker2006} provide a second-order regularizer
\[
R_2(w) = \int_{D} \left(\lvert \partial_{z_1z_1}w \rvert + \lvert  \partial_{z_2z_2}w \rvert \right)dz,
\]
which better fits the scenarios of this paper.

%The data term $y$ of depth images and motion (such as optical flow) usually  comprises of linear trend, hence piecewise affine models are more suitable. The work of \cite{Erdogan2012,Holz2012} have incorporated depth information on top of the RGB image, which largely improves RGB-D image segmentation accuracy when occlusion is present by modeling the depth image of slanted planars as linear functions. 

Let $[n]$ denote the discrete set $\{1,2,\ldots, n\}$. Given $n$ signals with coordinates $[n]$ and the $n$ dimensional data vector $\y$, the classical (discrete) piecewise constant Potts model~\cite{potts} has the form
\begin{equation}
\label{potts0}
\min_{\w} \;\|\w-\y\|_2 + \lambda\|\nabla^1 \w\|_0,
\end{equation}
where $\|\cdot\|_2$ denotes the $\ell_2$ norm, and $\|\cdot\|_0$ the $\ell_0$ norm. The \emph{discrete first derivative} $\nabla^1 \w$ of the fitting vector~$\w\in\R^n$ is the $n-1$ dimensional vector $(w_2-w_1,w_3-w_2,\ldots,w_n-w_{n-1})$ and the $\ell_0$ norm 
of a vector is its number of nonzero entries. The case of 2D images can be easily generalized.

Compared to  the TV regularization term which over-penalizes the
sharp discontinuities between two regions in an image, the $\ell_0$ term in the Potts model is more desirable, but also computationally costlier. 
The discrete Potts model is in general $\mathcal{NP}$-hard to solve. The work of~\cite{german} was one of the first to utilize the Potts, and recently~\cite{2018arXiv180307351S,10.1007} formulate it as a MIP that could find global optimum. %The authors of~\cite{Martin2017} propose efficient algorithms based on the Viterbi algorithm~\cite{Felzenszwalb2006} that compute Potts estimators of $\ell_1$ data terms.

%The Blake-Zisserman model is applied in~\cite{2013ISPAn} to a digital surface segmentation, and iterated methods based on simple thresholding are used. The work of~\cite{stereo2} approximates nearest neighbor of each pixel according to an affine plane in $3$D, it achieved top performance among all the local methods in stereo matching with slanted support windows. 
%Recently, \cite{2018Fortun} proposes a new method to estimate piecewise affine motion fields  directly without intermediate segmentation.

Apart from denoising, we also look into the segmentation problem. In graph based models, one first builds a square grid graph $G(V,E)$ to represent an image, where $V$ corresponds to pixels of an image grid and $E$ represents the $4$ or $8$ neighboring relations between pixels.

A graph partitioning $\cal{V}$ is a partition of  $V$ into disjoint node sets $\{V_1, V_2, \ldots, V_k\}$. And in graph-theoretical terms, the problem  of image segmentation corresponds to graph partitioning. The \emph{multicut} induced by $\cal V$ is the edge set 
$\delta(V_1, V_2, \ldots, V_k) = \{uv \in E \mid \exists i \ne j \text{ with } u\in V_i \text{ and } 
v \in V_j\}$. 
Hence, an image segmentation problem  can be represented either by \emph{node labeling}, i.e., assigning a label to each node $v\in V$, or by
 \emph{edge labeling}, i.e., a multicut defined by a subset of edges $E^\prime \subseteq E$, see the left image of Figure~\ref{multicut2} as an example, where the multicut of $8$ dashed edges uniquely defines a partition of the $4\times 4$-grid graph  into~$3$ segments.

In machine learning, one often distinguishes between \emph{supervised} and \emph{unsupervised} segmentation. In the former case, the labels of classes (e.g., person, grass, sky, etc) are pre-defined, and annotated data is needed to  train the model.
Among many existing supervised models, the classical Markov Random Field (MRF) is well studied, and interested readers may refer to~\cite{mrf-s} for an overview of this field. Recently, Deep Convolutional Neural Networks~\cite{0483bd94} (DCNN) have become increasingly important in many computer vision tasks, such as semantics and instance segmentation~\cite{7478072,Hayder2016BoundaryAwareIS}. However, huge amount of annotation effort (in terms of pixel level annotated data) and computational budget (in terms of number of GPUs and training time) are needed.

In the unsupervised case, the labels' class information is missing. This introduces ambiguities when node labeling is used. See for example the node labeling in Figure~\ref{multicut2}. If we permute the labels (colors), it will result in the same segmentation. On the contrary, edge labeling (e.g., by multicuts) does not exhibit such symmetries and is therefore more appealing in this case. Recent notable approaches
are the (lifted) multicut problems~\cite{kappesglobally,pmlr-hornakova17a,towards} based on  Integer Linear Programming (ILP) formulations, which label edges ($0$ or $1$) instead of pixels. The \emph{multicut constraints}~\cite{towards} (introduced in Section~\ref{section_multicut}) are used to enforce a valid segmentation. These methods do not require annotated data and can be run directly on CPUs.
In this paper, we will focus on this approach.
%The multicut approach is appropriate when one assumes that the input signals $y$ is approximately piecewise constant. Most of the literature deals with this problem type \cite{kappesglobally,andres2011,grabcut,towards,graphcut}.

In this work, we borrow ideas from the second derivative TV and Potts model, and propose a novel MILP formulation for the discontinuous piecewise affine fitting problem. 
%We will show later that the discrete version of~\eqref{potts} alone does not necessarily impose a unique or even valid segmentation.  The inclusion of the multicut constraints~\cite{kappesglobally} is then necessary to obtain consistent partitions.
The original contributions of this paper are as follows.
\begin{itemize}
\item We propose an approximate and non-parametric model for the general discontinuous piecewise affine fitting problem.
\item The model is formulated as a MILP and multicut constraints are added using cutting plane method to  ensure a valid segmentation.
\item The piecewise affine function can be easily constructed given the segmentation.
%\item It is unsupervised, hence no training data is needed.
\end{itemize}

%Although in this paper, segmentation is our main focus, information   reduction and denoising comes as a kind of nice side-effect of the affine transformation. 
%------------------------------------------
\section{MIP for the piecewise linear fitting model: 1D}
\label{sec:1dproblem}
We first restrict ourselves to the simple $1$D signals case where the signal domain $D\subseteq \Z^1$ (could be easily generalized to $D\subseteq \R^1$). Our model is able to find the optimal piecewise linear function $f\in\mathcal{F}$ that best fits the original data~$y$. 
%We prefer to work on the discrete space where the signals lie in $[n]$, and our MIP formulation is based on graph $G(V,E)$.

%---------------------
\subsection{Modeling as a MIP}
The $1D$ signals with $n$ discrete points could be naturally modeled as a chain graph. The associated graph $G(V,E)$ is defined with $V=\{i\;| \;i\in[n]\}$ and $E=\{e_{i}= (i,i+1)\;|\; i\in [n-1]\}$.
We introduce $n-1$ binary variables:
$$ x_{e}=\left\{
\begin{aligned}
1, & \;\;\text{if two end nodes of edge e are in different affine segments, }\\
0, & \;\;\text{otherwise,}
\end{aligned}
\right.
$$
where an edge $e$ is called \emph{active} if $x_e=1$, otherwise it is \emph{dormant}. 
%In this section, we will denote the binary variable $x_{e_i}$ as $x_i$ for simplicity.

%We denote $\x$ as the vector of $n-1$ binary variables~$x_{e_i}$ indicating  whether the edge $e_i$ is a jump or not.
Our goal is to fit a piecewise linear function $f\in\mathcal{F}$ to the input data $y$. We denote the fitting value $w_i := f(z_i)$, for $i\in[n]$. The coordinate $z_i=i$ and denote $x_{e_i}$ as $x_i$. We further define the following property:
\begin{align}
	\nabla^2w_i= 0  &\Leftrightarrow x_{i-1} =  x_{i} = 0, \; \; i\in[2:n-1],\label{logic}
\end{align} 
where $\nabla^2w_i:= w_{i-1}-2w_i+w_{i+1}$ is the the discrete second derivative, and $[2:n-1]$ denotes the discrete set $\{2,3,\ldots,n-1\}$.

The above property can be modeled via MIP using the ``big $M$'' technique, which leads to the  formulation
\begin{alignat}{2}
	\min \;\; \sum\optlimits_{i=1}^n&| w_i -y_i |+\lambda \sum\optlimits_{i=1}^{n-1} x_{i} \label{eq:Pott_1D2}\\ 
	|\nabla^2w_{i}| &\leq M(x_{i-1}+x_i), \;\;  i\in[2:n-1],\tag{\ref{eq:Pott_1D2}a}\label{bigmcons}\\
	w_i &\in \R,\;\; i\in[n],\nonumber\tag{\ref{eq:Pott_1D2}b}\\
	x_{i} &\in \{0,1\},\;\;i\in[n-1]\nonumber\tag{\ref{eq:Pott_1D2}c},
%& Z_i \in \{0,1\}, \forall i=1,\ldots, n-2.
\end{alignat}
where $\lambda>0$ is similar to the regularization term in the Potts model~\eqref{potts0}. 
It is worth to mention that there are common tricks to formulate~\eqref{eq:Pott_1D2}-(\ref{eq:Pott_1D2}c) as a MILP. Namely, $|w|\leq Mx$ is replaced by two constraints $w\leq Mx$ and $-w\leq Mx$, and the absolute term $|w-y|$ in the objective function is replaced by $\epsilon^+ +\epsilon^-$, plus an additional constraint $w-y = \epsilon^+ -\epsilon^-$, where $\epsilon^+ \geq 0$, $\epsilon^- \geq 0$.
% and $M$ should be no less than $\text{max}\;\{|\nabla^2w_{i}|\;|\; i \in [2:n-1]\}$  so that \eqref{bigmcons} is always valid. We will show in Section~\ref{techniques} how to find a suitable value. 

\begin{lemma}
\label{lemma1}
The optimal solution $x^\star$ of problem~\eqref{eq:Pott_1D2}-(\ref{eq:Pott_1D2}c) satisfies property~\eqref{logic}.
\end{lemma}

\begin{proof}
The direction that $x^\star_{i-1} =  x^\star_{i} = 0 \Rightarrow \nabla^2w_i= 0$  directly follows constraint~\eqref{bigmcons}.
On the other hand, if $\nabla^2w_i = 0$, the optimal solutions satisfy $x^\star_{i-1}+x^\star_{i} = 0$ (thus $x^\star_{i-1} =  x^\star_{i} = 0$) since~\eqref{eq:Pott_1D2}-(\ref{eq:Pott_1D2}c) is a minimization problem with positive weights on~$x$.  
\end{proof}

Figure~\ref{fig:potts1D} shows an example of $3$ affine segments and $2$ active edges computed by formulation \eqref{eq:Pott_1D2}-(\ref{eq:Pott_1D2}c). We see that the optimal solution $w_i$ is the fitting value for node $i$, and $x_i=1$ acts as the boundary between two affine segments. As a result, the nodes between two active edges define one segment, and the signals within one segment share the same linear slope. Although being non-parametric, the linear parameters for each segment can be easily computed afterwards, and the number of segments equals $\sum\optlimits_{i=1}^{n-1} x_{i}+1$. Hence, upon solving the MIP formulation~\eqref{eq:Pott_1D2}-(\ref{eq:Pott_1D2}c) in $1D$, a piecewise linear function $f\in\mathcal{F}$ can be easily constructed, and $w_i := f(z_i)$, $\forall i\in[n]$.

\begin{figure}[!t]
	\centering
	\includegraphics[width=0.7\linewidth]{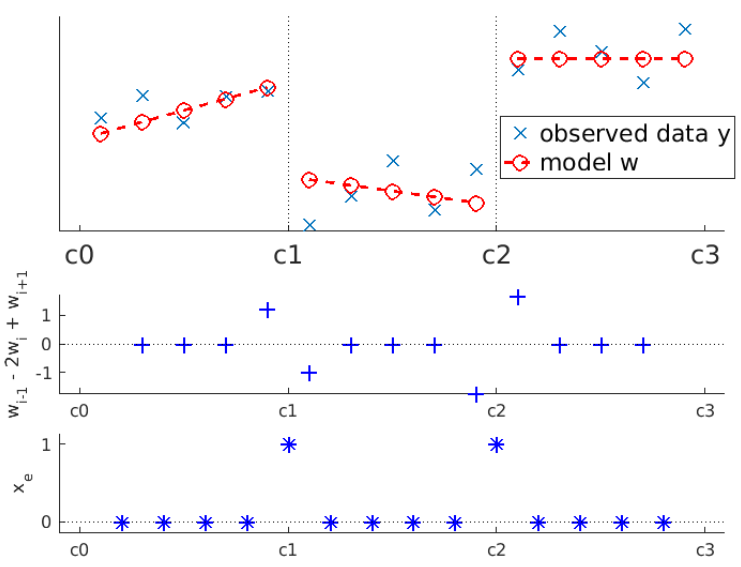}
	\caption[An example with $3$ affine segments and $2$ active edges]{An example with $3$ affine segments and $2$ active edges. }
	\label{fig:potts1D}
\end{figure}

Note in the above example, the cases where $\nabla^2 w_i\neq0$ actually induces $x_{i-1}+x_i=1$, for some $i\in [2,n-1]$. However, there exists instances where $x_{i-1}+x_i=2$ for $\nabla^2 w_i\neq0$. The image on the left of Figure~\ref{outlier} depicts an example where the node~$5$ is an outlier (as an one node segment), and $x_{e_l}+x_{e_r}=2$.
We also observe that problem~\eqref{eq:Pott_1D2}-(\ref{eq:Pott_1D2}c) does not necessarily output unique optimal integer solution $x$. One extreme example is shown in the right image of Figure~\ref{outlier}, where either $x_{e_l}$ or $x_{e_r}$ can be active (but not both), and they yield the same optimal objective value.

%we can conclude that~\eqref{eq:Pott_1D2} indeed corresponds to a piecewise linear fitting model.

\begin{figure}[h]
	\centering
	\includegraphics[width=0.48\linewidth]{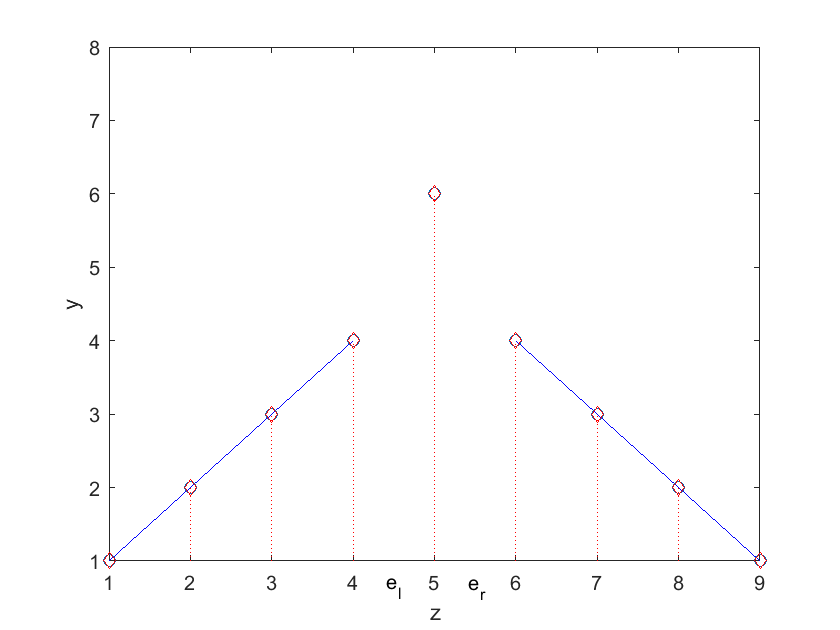}
	\includegraphics[width=0.48\linewidth]{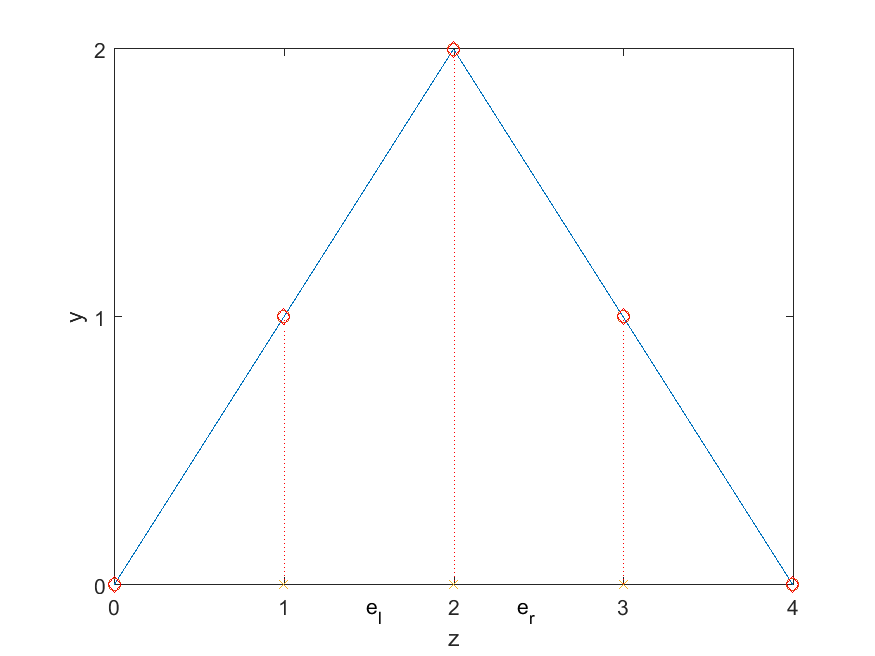}
	\caption{Left: example where outlier exists (both $e_l$ and $e_r$ are active). Right: example with two segments where the optimal solution is not unique (either $e_l$ or $e_r$ is active).}
	\label{outlier}
\end{figure}

 %In this case, the piecewise linear function $f$ is continuous, and the node at position $2$ can be either in the first or second partition.

%Note if there is prior knowledge about the number of segments, then the following linear constraints can be added to \eqref{eq:Pott_1D2},  which states that $w$ has to be piecewise linear with at most $k$ jumps. 
%\begin{align}
%	 \sum\optlimits_{i=1}^{n-1} x_{e_i}&\leq k.\label{upper}
%\end{align}

%Similar ideas can be applied to $\w$, which will be discussed in details in Sec.~\ref{techniques}.

%------------------------------------------
\section{MIP of the piecewise affine fitting model: 2D}
\label{sec:2dproblem}

We are more interested in the $2$D image case where the domain $D\subseteq \Z^2$. Our model is able to find the fitting value $w$ and a valid segmentation. The optimal piecewise affine function can be approximated and constructed based on the segmentation. 

%------------------
\subsection{Modeling as a MIP}
A $2D$ image with $m\times n$ pixels could be naturally modeled as a square grid graph $G(V,E)$, where $V=\{(i,j)| \;i\in[m], j\in [n]\}$, and $E$ represent the relations between the center and its $4$ neighboring pixels (see Figure~\ref{multicut2} for demonstration).
Let $z_{i,j}=(i,j)\in \Z^2\subseteq D$ be the coordinates for pixel $(i,j)$, and the matrix $Y=(y_{i,j})\in \R^{m\times n}$ be the intensity values of the image. 
%The piecewise affine fitting problem then corresponds to simultanously find the segments and approximates a linear function within each segment. $f\in \mathcal{F}$.
We divide the edge set $E$ of the grid graph into its horizontal (row) edge set $E^r$ and its vertical (column) edge set $E^c$. So $E = E^r \cup E^c$, and   $E^r \cap E^c = \emptyset$.
Denote $e_{i,j}^r\in E^r$ to present edge $((i,j),(i,j+1))$ and $e_{i,j}^c\in E^c$ to represent $((i,j),(i+1,j))$. Again for simplicity, we denote the binary edge variables $x^r_{i,j}:= x_{e^r_{ij}}$ and $x^c_{i,j}:= x_{e^c_{ij}}$.

The piecewise affine fitting model in $2$D is obtained by formulating \eqref{eq:Pott_1D2}-(\ref{eq:Pott_1D2}c) per row and column
\begin{align}
	\min\;  \sum\optlimits_{i=1}^m &\sum\optlimits_{j=1}^n | w_{i,j} - y_{i,j} |+\lambda \sum\optlimits_{e\in E} x_{e} \label{TV2}\\
	|\nabla_r^2w_{i,j}| &\leq M(x^r_{i,j-1} + x^r_{ij}), 	    \;\; i \in [m],\; j\in[2:n-1],\tag{\ref{TV2}a} \label{TV21}\\
	|\nabla_c^2w_{i,j}| &\leq M(x^c_{i-1,j} + x^c_{ij}), \;\; j\in[n],\; i\in [2:m-1],\tag{\ref{TV2}b} \label{TV22}\\
w_{ij} &\in \R, \;\;\;\;\; i\in[m], \;  j\in[n],\tag{\ref{TV2}c}\\
x_e  &\in \{0,1\}, \;\;e\in E \tag{\ref{TV2}d},
\end{align}
where $M$ is again the big-M constant. Here, $\nabla_r^2w_{i,j}=w_{i,j-1}-2w_{i,j}+w_{i,j+1}$, and $\nabla_c^2w_{i,j}=w_{i-1,j}-2w_{i,j}+w_{i+1,j}$. That is, the discrete second derivative with respect to $z_1$ and $z_2$-axis.
%To the best of our knowledge, formulation~\eqref{TV2} has not been studied before in the literature. 
Upon solving~\eqref{TV2}-(\ref{TV2}d), it serves for the purpose of denoising by computing $w$. But two questions still remain: does the binary solution $x$ represent a valid segmentation? If so, is the corresponding piecewise affine function $f\in\mathcal{F}$ (obtained by affine fitting each segment) aligned with $w$, i.e., is $w_{ij} = f(z_{i,j})$?

The answers to both questions are ``no'', unfortunately. We will show in the next two sections that, the first one could be fixed by enforcing the multicut constraints. But the second one is not guaranteed, thus making our model approximate.

%there is the issue that the segments formed by this model may not have closed contours~\cite{andres2011}. 

%------------------
\subsection{Multicut constraints for valid segmentation}
\label{section_multicut}
The multicut constraints introduced in~\cite{andres2011} are inequalities that enforce valid segmentation in terms of edge variables. It reads

\begin{equation}
\sum_{e\in C\setminus\{e'\}}x_e\geq x_{e'}, \;\;  \forall \;\text{cycles $C\subseteq E$, $e'\in C$}, \label{multicut_a}
\end{equation}
which basically says that for any cycle, the number of active edges cannot be $1$. Recall that an edge is called active if its two end nodes belong to different segments. Because otherwise, the two nodes of the active edge are again ``linked'' (hence belong to the same segment) by connecting the rest edges of the cycle, hence a contradiction.

We now prove the following lemma.

\begin{lemma}
\label{lemma4}
The multicut constraints~\eqref{multicut_a} are needed for the optimal solution~$x$ of \eqref{TV2}-(\ref{TV2}d) to form a valid segmentation. 
\end{lemma}

\begin{proof}
We prove this lemma by constructing a counter-example as follows:

In the left image of Figure~\ref{badcut}, the data terms $y$ of all $15$ pixels are constructed to lie exactly in two affine planes with respect to their coordinates $z=(z_1,z_2)$. The optimal affine function of the left plane is $y=4-z_2$ and the right one is $y=z_2$. We shall see that the $3$ pixels with data $y=2$ lie on both affine planes with respect to the coordinates $z$.

\begin{figure}[t]
\centering
\includegraphics[width=0.42\columnwidth]{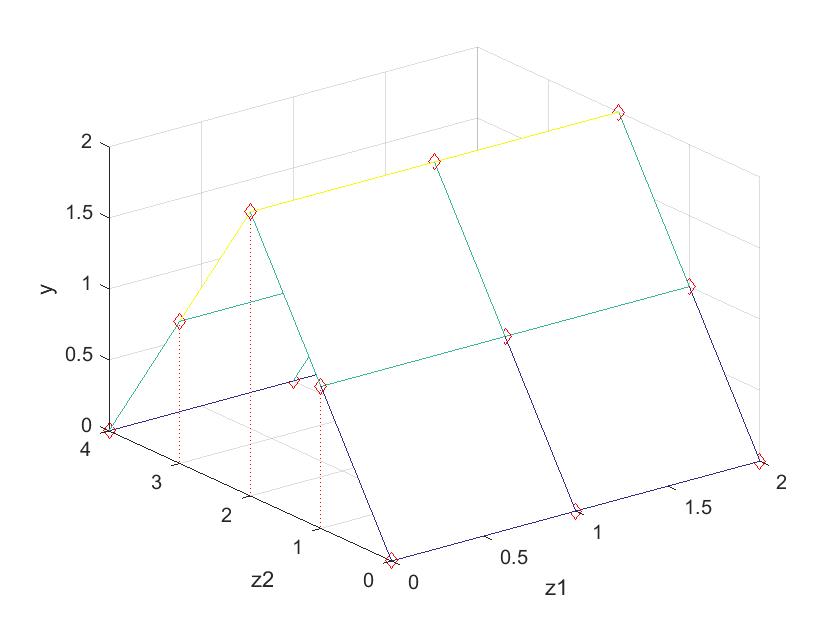}
\includegraphics[width=0.33\columnwidth]{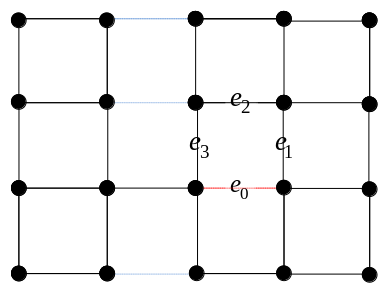}
\caption[A counter-example where model~\eqref{TV2}-(\ref{TV2}d) does not form a valid segmentation.]{A counter-example where model~\eqref{TV2}-(\ref{TV2}d) does not form a valid segmentation. Left: $3$D view of input image. Right: the corresponding graph and active edges.}
\label{badcut}
\end{figure}

%If we project  each column of the pixels in $3D$ into the $z_1,y$-space, the intensities $y$ of each five pixels lie in the same straight line, hence no jumps occur.
If we project the $3$D plot into the $z_2,y$-space, for every row of the image grid, it is exactly the same $1$D case we studied in the right image of Figure~\ref{outlier}. We have showed there that the optimal solution is not unique. 

Hence, we can easily construct one optimal solution $x^\star$ ($3$ blue edges plus $1$ red edge) of~\eqref{TV2}-(\ref{TV2}d) shown in the right image of Figure~\ref{badcut}, where the multicut constraints~\eqref{multicut_a} is not satisfied. That is,  there exists a cycle $e_0$-$e_1$-$e_2$-$e_3$ that violates it.
\end{proof}

%As a result, the multicut uniquely define an image segmentation, and each maximal set of vertices with dormant edges corresponds to one segment. 

%Note that in the above illustrated example, the optimal solution has the same optimal objective value $3\lambda$ with the counter-example. This is often referred to as \emph{symmetry} in combinatorial optimization, and hence the multicut constraints~\eqref{multicut_a} are, in a sense, \emph{symmetry-breaking} inequalities with respect to problem~\eqref{TV2}.

%------------------
\subsection{The main formulation in $2$D}
We thus need to add the multicut constraints~\eqref{multicut_a} to the piecewise affine fitting model~\eqref{TV2}-(\ref{TV2}d), to  form a valid segmentation. This leads to the main formulation of our paper

\begin{align}
	\min\;  (1-\lambda)\sum\optlimits_{i=1}^m \sum\optlimits_{j=1}^n &| w_{i,j} - y_{i,j} |+\lambda \sum\optlimits_{e\in E} x_{e}\label{model2}\\
	|\nabla_r^2w_{i,j}| &\leq M(x^r_{i,j-1} + x^r_{ij}), 	    \;\; i \in [m],\; j\in[2:n-1],\tag{\ref{model2}a} \label{TV21}\\
	|\nabla_c^2w_{i,j}| &\leq M(x^c_{i-1,j} + x^c_{ij}), \;\; j\in[n],\; i\in [2:m-1],\tag{\ref{model2}b} \label{TV22}\\
	\sum\optlimits_{e\in C\setminus\{e'\}}x_e&\geq x_{e'},  \;\;\forall \;\mbox{cycles} \; C\subseteq E, e'\in C,\label{multic}\tag{\ref{model2}c}\\
w_{ij} &\in \R, \;\; i\in[m], \;  j\in[n],\tag{\ref{model2}d}\\
x_e  &\in \{0,1\}, \;\;e\in E \tag{\ref{model2}e}.
\end{align}

Note that the number of inequalities~\eqref{multic} is exponentially large~\cite{andres2011} with respect to $\lvert E \rvert$, where $\lvert E \rvert$ denotes the number of edges in $G$. Hence, in practice, it is not possible to include them into~\eqref{model2}-(\ref{model2}e) at one time. We will discuss in details in Section~\ref{exact_bc} the cutting plane algorithm that handles~\eqref{multic}.

It is well known that if a cycle $C\in G$ is chordless,
then the corresponding multicut constraint~\eqref{multicut_a} is facet-defining for the corresponding multicut 
polytope~\cite{kappesglobally,KAPPES2016}. Among all, the simplest ones of a grid graph are the $4$ and $8$-edge chordless cycle constraints (see the 4-edge cycle $e_0-e_1-e_2-e_3$ in Figure~\ref{multicut2} for an example), and the number of these constraints are linear to $\lvert E \rvert$.
\begin{figure}[t!]
 \centering
 \includegraphics[width=0.35\linewidth]{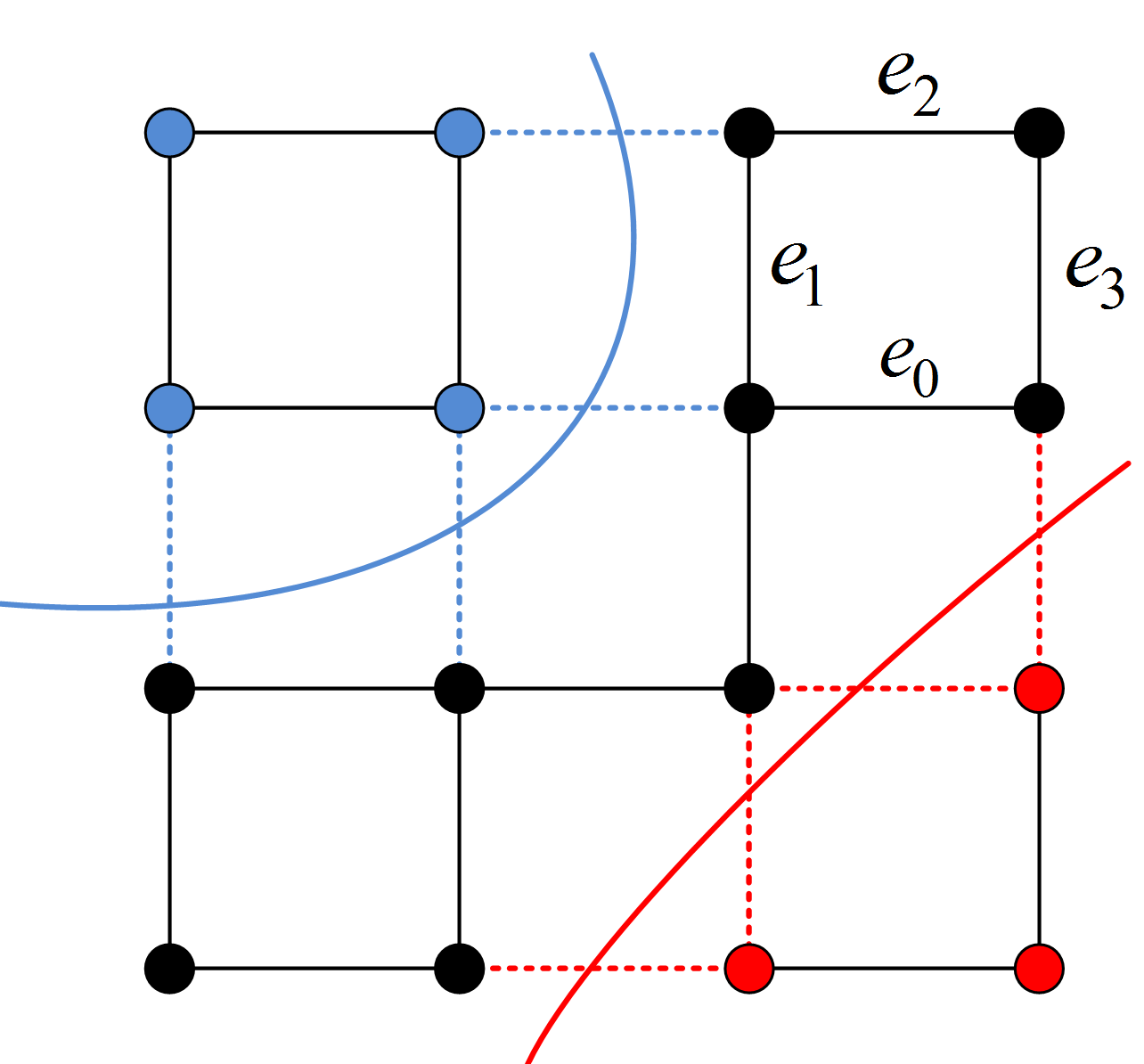}
 \includegraphics[width=0.3\columnwidth]{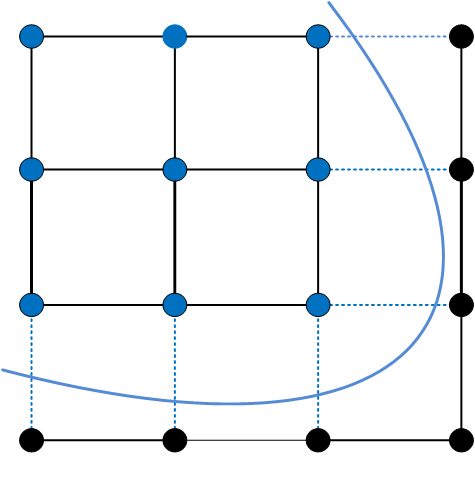}
 \caption[Two representations of image segmentation.]{Left: two representations of an image segmentation: node labeling (by colors) and edge labeling via multicuts (dashed edges). Right: example of a $9$-pixel segment.}
 \label{multicut2}
\end{figure}
In Section~\ref{Computations2}, we will test different strategies of adding the $4$ and $8$-edge chordless cycle constraints 
to~\eqref{model2}-(\ref{model2}e) as initial constraints.

%------------------
\subsection{Approximate model for piecewise affine fitting}
Finally, we prove the following theorem.

\begin{figure}[t]
\centering

\label{cut_edge}
\end{figure}

\begin{theorem}
\label{lemma5}
The MIP formulation~\eqref{model2}-(\ref{model2}e) is only approximate in finding the optimal piecewise affine fitting function $f \in {\cal F} $ that best fits $y$, i.e., \eqref{potts1} does not hold.
\end{theorem}

\begin{proof}
We prove this theorem by constructing a counter-example where the optimal solution $w^\star$ of~\eqref{model2}-(\ref{model2}e) within one segment does not lie in any affine function $f^i$ with respect to the coordinates $z$.

We construct an optimal solution $x^\star$ which corresponds to the segmentation in the right image of Figure~\ref{multicut2}, where the $9$ nodes on the top left corner form a segment. We restrict ourselves to this segment where the integer coordinates of the pixels range from $(0,0)$ to $(2,2)$.

By constraint~\eqref{TV21}, the $w^\star$ of the $3$ nodes on each row satisfy the same linear function. Assume the linear function in the first and second row of nodes satisfy  $w=a_1z+b_1$ and $w=a_2z+b_2$, where $(a,b)$ are the linear parameters and $z$ the discrete coordinates that range from $0$ to $2$ in this case. Then the fitting value $w^\star$ of the $6$ nodes on the first two rows are listed in the following matrix:

\[
\begin{bmatrix}
    w_{00}& w_{01} & w_{02}  \\
    w_{10} & w_{11} & w_{12}
\end{bmatrix}=
\begin{bmatrix}
    b_1 & a_1+b_1 & 2a_1+b_1  \\
    b_2 & a_2+b_2 & 2a_2+b_2
\end{bmatrix}.
\]

We can then compute $w_{22}$ using constraint~\eqref{TV22}, where $w_{22} = 2w_{12} - w_{02}=4a_2+2b_2-2a_1-b_1$.
We note that if $w^\star$ of the $9$ nodes lies in any affine function $f^i$, then $
w_{00}-2w_{11}+w_{22} = 0.
$

However, we have
$w_{00}-2w_{11}+w_{22} = 2(a_2-a_1),
$ which is a contradiction when  $a_1\neq a_2$.
Thus we complete the proof.

\end{proof}

Although the MILP formulation~\eqref{model2}-(\ref{model2}e) is not exact on solving~\eqref{potts}, we still get a valid segmentation. We can then fit an affine function within each segment afterwards, thus obtaining a valid (although not optimal) piecewise affine function $f\in\mathcal{F}$ as post-processing.

%------------------------------------------
\section{Solution Techniques}
\label{techniques}
We now introduce a heuristic and an exact algorithms to solve \eqref{model2}-(\ref{model2}e).

%------------------ 	
\subsection{Region fusion based heuristic algorithm}
\label{regionfusion2}
The resulting problem~\eqref{model2}-(\ref{model2}e) is a MILP, which is solved using any off-the-shelf commercial MIP solvers. The underlying sophisticated algorithms are based on the branch and cut algorithm, where a good global upper bound usually helps to improve the performance. In the following, we will introduce a fast heuristic algorithm that provides a valid segmentation. It was then given to~\eqref{model2}-(\ref{model2}e) and upon solving a linear program, its solution is served as a global upper bound.

Our heuristic is based on the region fusion algorithm~\cite{Fast} which approximates the Potts model~\eqref{potts0}. 
We start by performing parametric affine fitting over the $4$ groups ($2\times 2$ squared nodes) of each node, as shown in Figure~\ref{Regions}. We take the group that has the minimum fitting MSE, and assign the affine parameters (a vector of $3$ in $2$D case)  to that node. Note that nodes located on the boarders of the grid graph only have $2$ such groups, while corner nodes only have $1$ group.
\begin{figure}[t]
\center
\includegraphics[width=0.4\linewidth]{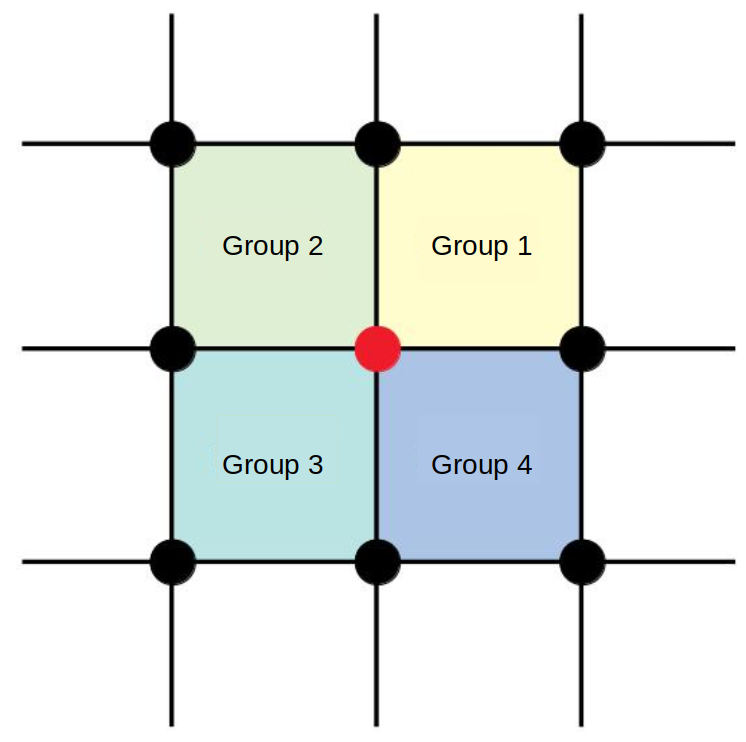}
\caption{Each node (colored red) has $4$ groups of $2\times 2$ squared nodes for affine fitting.}
\label{Regions}
\end{figure}
Our algorithm then starts with every node $i$ belonging to its own segment $V_i$, and for each pair of nodes, the following minimization problem is solved.
\begin{equation}
\min_{\w} \;\tau_i\left\Vert w_i - Y_i \right\Vert_2 + \tau_j\left\Vert w_j - Y_j \right\Vert_2+\kappa_t \gamma_{ij} \mathds{1}( w_i \neq w_j ), \label{l_0_region}
\end{equation}
where  $\mathds{1}(\cdot)$ denotes the indicator function, $\tau_i$  the number of nodes in segment $V_i\subseteq V$, and $\gamma_{ij}$ represents the
number of neighboring nodes between two segments $V_i$ and $V_j$. Here, $Y_i$ indicates the affine parameter of segment $V_i$, and $w_i$ the unknown variables, and $\kappa_t$ express the regularization parameter at the $k_{th}$ round of iteration.

To speed up computation, instead of  solving~\eqref{l_0_region} exactly, the following criteria is checked instead (see~\cite{Fast} for more detailed description):
  $$
\tau_i\tau_j \lVert Y_i-Y_j \rVert_2\leq \kappa \gamma_{i,j} \left( \tau_i + \tau_j \right).
$$
If the above condition holds, we merge segment $V_i$ and $V_j$, and the updated affine parameter (also the values of $w_i$ and $w_j$) is obtained by conducting a parametric affine fitting over the new segment. If not, the two segments and their affine parameters stay the same.

The algorithm iterates over each pair of nodes for solving~\eqref{l_0_region}, and the regularization parameter $\kappa$ grows over every round of iteration, which increasingly encourages merging. The algorithm stops after $t$ round of iteration, when $\kappa_t = \lambda$, where $\lambda$ is the pre-defined regularization parameter with respect to~\eqref{potts0}.

%------------------ 	
\subsection{Exact branch and cut algorithm}
\label{exact_bc}
Apart from the classical branch-and-cut algorithm inside the MIP solver, we describes below the cutting plane method that iteratively add lazy constraints from~\eqref{multic}.

\textbf{Cutting plane method}.
Similar to the cutting planes method that solves the multicut problem~\cite{kappesglobally}, we start solving~\eqref{model2}-(\ref{model2}e) by ignoring constraints~\eqref{multic}, or with few of them (e.g., the $4$ or $8$-edge cycle constraints). %In each step we solve a problem, check the feasibility, detect violated constraints from the current separation problem, and augment corresponding violated constraints (cuts) to~\eqref{model2}-(\ref{model2}e). This procedure is repeated until no more violated constraints are found.

We then check the feasibility of the resulting solution  with respect to~\eqref{multic}. If it is already feasible, we are done and the optimal solution to~\eqref{model2}-(\ref{model2}e) is achieved. Otherwise, we identify the current separation problem and then add the corresponding violated constraints (cuts) to~\eqref{model2}-(\ref{model2}e). We resolve the updated MILP, and this procedure repeats until either we get the optimal solution, or the user-defined limit is reached.

%In a grid graph, although the number of constraints~\eqref{multic} is still exponential, it may be advantageous to add the 4-edge cycle constraints as initial constraints (as discussed in the piecewise constant Potts). We conduct some computational experiments with and without them later.

\textbf{Separation problem}. 
Given an integer solution, it is polynomial to either check the feasibility with respect to~\eqref{multic}, or to identify and separate the integer infeasible solutions by adding violated constraints.

Phase $1$: Given the incumbent solution of the MILP~\eqref{model2}-(\ref{model2}e), we extract its binary solutions and remove edges where $x_e=1$ from the grid graph $G(V,E)$.  We thus obtain a new graph $G^\prime(V^\prime,E^\prime)$ where $V^\prime = V$, $E^\prime \subseteq E$ and we identify its connected components. We then check for each active edge to see if their two end nodes belong to the same component. If there exists any, the current solution is infeasible (and we call the corresponding active edges violated). Otherwise, a feasible and optimal solution is found.

Phase $2$: If violated edges exist, we search for violated constraints by finding paths between the two nodes of the edge. We first conduct a depth-first search on the graph $G^\prime$, and multiple such paths could be found. We set the maximum depth to $10$ to restrict the searching time. If the depth-first search does not return any path, we then switch to the breadth-first search to return only one shortest path. 

Phase $3$: For each violated edge, we add the corresponding multicut constraints~\eqref{multic} (possibly many) to our MILP~\eqref{model2}-(\ref{model2}e), where the left hand side corresponds to the paths found in phase $2$ .

\textbf{Facet-defining searching strategy}. 
The above mentioned strategy that finds violated constraints does not guarantee facet-defining inequalities. Recall that the multicut constraint~\eqref{multic} is facet-defining if and only if the corresponding cycle is chordless. In the facet-defining searching strategy, we in addition keep track of the non-parental ancestors set (denoted $S$) of the current node during search. When we search for the next node, we  make sure that the potential node does not form an edge (with respect to $G$) with any node in $S$.

%Finally, the cutting plane method could be also applied on fractional solutions, where the user needs to set a threshold so that the edges above it are considered active and vise versa. This way, more cuts will be added. 
%A detailed study is beyond the scope of this chapter and we leave it for future research.

%------------------------------------
\section{Computational Experiments}
\label{Computations2}
In this section, all the experiments are conducted on a desktop with Intel(R) Xeon(R) CPU E5-2620 v4 @ 2.10GHz CPU and 64 GB memory, using IBM ILOG Cplex V12.8.0 as the MIP optimization solver.

We develop and compare the following variants of~\eqref{model2}-(\ref{model2}e) and report their computational results. The experiments are based on synthetic images of different sizes, as well as real depth images. We normalize the intensity values of all images to $[0, 1]$, and  each experiment is conducted $3$ times and only the median of the results is reported. We report the running time, nodes of the branch and bound tree, optimality gap, cuts added and the objective function of the MILP.

\begin{itemize}
	\item \textbf{MP}: The MILP formulation of the piecewise affine fitting model~\eqref{model2}-(\ref{model2}e) that adds the multicuts without the facet-defining searching strategy.
	%\item \textbf{MP-LP}: The linear programming relaxation of problem MP.
	\item \textbf{MPH}: MP where we adopt the solution of our heuristic as an initial input.
	\item \textbf{MPH-4}: MPH with the $4$-edge cycle multicut constraints as initial inequalities.
	\item \textbf{MPH-4\&8}: MPH  with the $4$ and $8$-edge cycle multicut constraints.
	\item \textbf{MPH-F}: MPH with the facet-defining searching strategy.
	%\item \textbf{MC-C}: Problem MP-4C with cardinality constraints introduced in Section~\ref{cardi_b}.
	%\item \textbf{MC-B}: MC with bounding constraints on the fitting value $w$.
	%\item \textbf{MC-M}: MC where we use different parameter to compute $M$ automatically.
	%\item \textbf{TGV}: The total generalized variation model~\cite{Total} (we use the second derivative).
\end{itemize}

%------------------ 	
\subsection{Automatic computation of parameters}
%In the computational experiments, we adopt the SOS-1 type constraint in Gurobi, which takes care of choosing the right M for us. 
\textbf{Parameter} $\bm{\lambda}$ is the regularization term employed to avoid over-fitting in problem MP~\eqref{model2}-(\ref{model2}e). We set~$\lambda$  independently for each row and column, denoted $\lambda_{i}^r$ and $\lambda_{j}^c$, since intuitively, this may help adapt to local features. 
%Consider the extreme case of a perfect linear data set, plus with one outlier. We can either choose to introduce two jump edges to perfectly fit the data, or just fit one linear line to the data (with the fitting error less than or equal to $|\nabla^2 Y^{r*}_i|$). 
$\lambda$ is computed in a way to avoid making an outlier a one-node segment. 
Let $\lambda^r_i=\frac{1}{2}\xi \cdot \text{max}_i \;|\nabla^2 y^r_i|$ and $\lambda^c_j=\frac{1}{2}\xi \cdot \text{max}_j \;|\nabla^2 y^c_j|$, where $\xi$ is the user-defined parameter.
In this manner, if there exists an outlier $(i,j)$, making a one-node segment will active all four edges of $(i,j)$, thus incurring a penalty value of $2(\lambda^r_i+\lambda^c_j)$.

\textbf{Parameter} $\bm{M}$ is for the ``big M'' constraint in MP~\eqref{model2}-(\ref{model2}e). In principle, it should be big enough so that the  constraints~(\ref{TV21},\ref{TV22}) are always valid, i.e., $M = 2$. On the other hand, it should be not too big, or it may harm the tightness of the LP relaxation. The value of big M could be computed automatically each on row and column, following the strategy above.
However, we have tested different variants and found out the results only have slight fluctuations. Hence, we simply set $M=2$ globally.
%Denote $M^r_i$ as the big $M$ constant for row $i$, and $M^c_j$ for column $j$, further denote $y^r_i$ be the vector of $y$ values in row $i$, and $y^c_j$ for column $j$.
%We compute $M$  separately for every row and column, and set $M^r_i$ to the largest value of $|\nabla^2 y^r_i|$ for the row $i$, and  $M^c_j$ to the maximum of $|\nabla^2 y^c_j|$. Here, the absolute functional of a vector $|(x_1,x_2,\ldots,x_n)|:=(|x_1|,|x_2|,\ldots,|x_n|)$. We further multiply it by an user defined factor $\xi_1$ to make $M$  big enough. So $M^r_i=\xi_1 \cdot\text{max}_i \;|\nabla^2 y^r_i|$, and $M^c_j=\xi_1 \cdot\text{max}_j \;|\nabla^2 y^c_j|$.

%For adding the Let $I^{\max}_{i,j}$ be the maximum intensity of the neighboring $10\times 10$ pixels, including $y_{i,j}$ itself, and  define $I^{\min}_{i,j}$ respectively, where $10$ is chosen large enough to ensure feasibility. 
\begin{figure}[t]
\centering
\includegraphics[width=0.32\columnwidth]{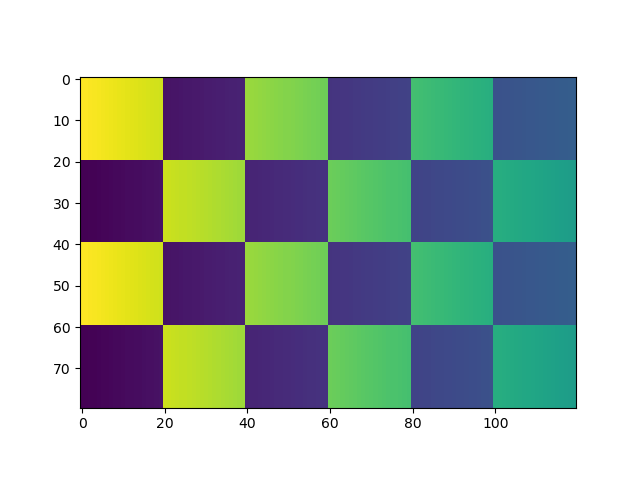}
\includegraphics[width=0.32\columnwidth]{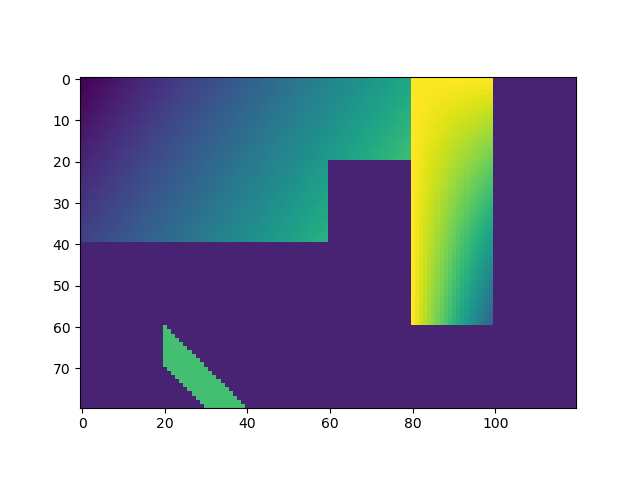}
\includegraphics[width=0.32\columnwidth]{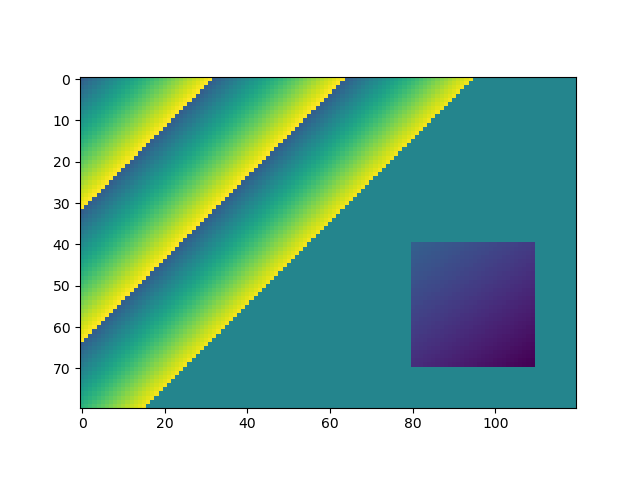}\\
\includegraphics[width=0.32\columnwidth]{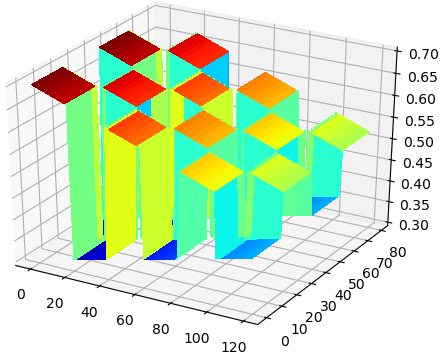}
\includegraphics[width=0.32\columnwidth]{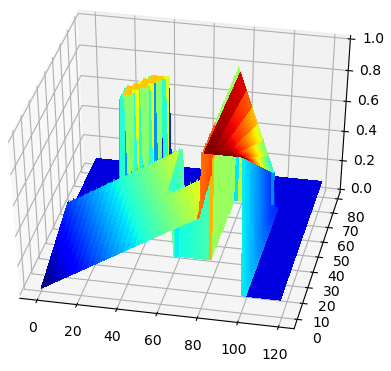}
\includegraphics[width=0.32\columnwidth]{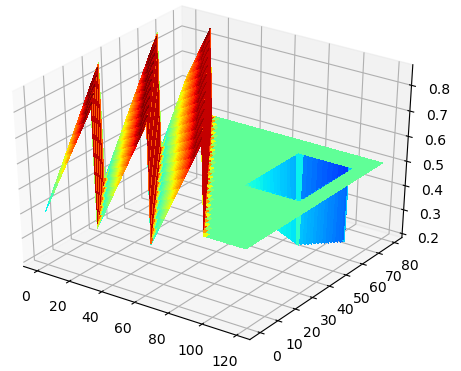}
\caption{Top: synthetic images with affine pieces, 2D view. Bottom: Their 3D views.}
\label{gray}
\end{figure}

%------------------ 	
\subsection{Detailed comparison on synthetic images}
%------------------ 	
%\subsubsection{MP vs MP-LP}
%We first conduct computational experiments on solving MP-LP against MP. When it comes to the separation problem of MP-LP, we will treat any edge with $x_e > 0.0001$ as an active edge. We set the parameter $M$ to $1$, $L$ to $0$ and $U$ to $2$ to ensure feasibility of~\eqref{model2}-(\ref{model2}e). We further set $\xi_2= 1$ for computing $\lambda$. The running time of MP-LP for the $2$ images are: $0.10$ and $0.07$ second. The percentage of fractional solutions for the $2$ images are: $74.5\%$ and $75.9\%$.

%As a result, sophisticated rounding methods must to be applied as post-processing, hoping to get an integer feasible solution of the Potts model~\eqref{model2}-(\ref{model2}e). We report the  analysis of MP in the next section.

In this section, we generate $3$ synthetic images that has affine trends, as shown in Figure~\ref{gray}. We then test different variants of our models on $3$ sizes of the images, i.e., $20\times 30$, $40\times 60$, and $80\times 120$. In addition, we further experiments on scenarios that add Gaussian noise of level $0$, $0.001$ and $0.005$. Thus, a total of $27$ tests ($81$ experiments, as we run each test $3$ times and only report the medium) are done for each model. We set the time limit of each experiment to $600$ seconds.

Before starting these $81$ experiment, we run additional experiments to select the ``right'' values of $\xi$. Since all three images already output optimal segmentation (with respect to the ground truth) results when $\xi =0.5$, we keep it fixed throughout this section to keep our comparison concise.

\begin{figure}[t]
\centering
\includegraphics[width=0.99\columnwidth]{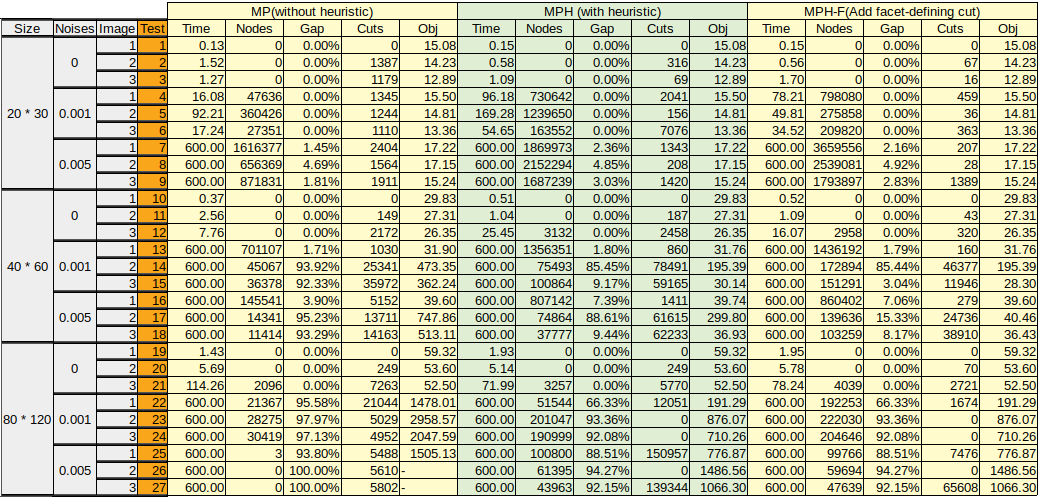}
\caption{Table on MP, MPH and MPH-F.}
\label{ex1-1}
\end{figure}

%------------------ 	
\subsubsection{MP vs MPH}
We first conduct experiments on solving MP with and without the heuristic algorithms  (introduced in Section~\ref{regionfusion2}) to the MIP solver. Our heuristic algorithm is fast to compute, takes $3$ seconds on average to converge on the $40\times 60$ sized images. Note that we only provide the MIP solver with initial integer solutions $x$ of problem~\eqref{model2}-(\ref{model2}e), hence it takes time for the solver to compute $w$ by solving a linear program.  

As we can see in the MP column of Figure~\ref{ex1-1}, MIP alone suffices to find optimal solutions in all tests when the image is clean (without Gaussian noise), even in $80\times 120$ size. It also reaches optimality on the $20\times 30$ images, with $0.001$ Gaussian noise added. However, without heuristic, no feasible solution are found in Test $26$ and $27$ within $600$ seconds.
The results in MPH column indicates that adding the result of the heuristic as initial solution to the MIP solver mostly improves the results. For instance, MPH helps reduce the optimality gap from $92.33\%$ to $9.17\%$ in Test $15$. It sometimes also reduce the performance, i.e., increases the running time of finding optimal solution from $16.08$ to $96.18$ seconds in Test $4$.

%------------------ 	
\subsubsection{MPH vs MPH-F}
Given an heuristic solution, we further test the performance of adopting the facet-defining searching strategy. Recall that although it takes more time to find a facet-defining multicut constraint~\eqref{multic} (as described in the facet-defining searching strategy), it is tighter compared to non facet-defining ones.
The results are shown in the MPH and MPH-F columns of Figure~\ref{ex1-1}, where we could see MPH-F performs better than MPH in most of the cases, with only a few exceptions. For instance, MPH-F helps reduce the running time from $169.28$ to $149.81$ seconds in Test $5$. MPH-F also reduces the optimality gap from $88.61$\% to $15.33$\% in Test $17$.

%We report that it takes $1.1$ and $0.8$ seconds for the solver to get feasible solutions with objective value $416.9$ and $129.4$. Moreover, the MIP solver is able to search for new feasible solutions based on the provided initial solutions. We report that the MIP solver found better solutions with  objective $142.9$ and $52.6$ in $5.2$ and $4.6$ seconds on the first and second image, respectively.

%The segmentation results of MC-H based on initial solutions computed by our region fusion based heuristic  are plotted in Figure~\ref{image:mch}.

\begin{figure}[t]
\centering
\includegraphics[width=0.99\columnwidth]{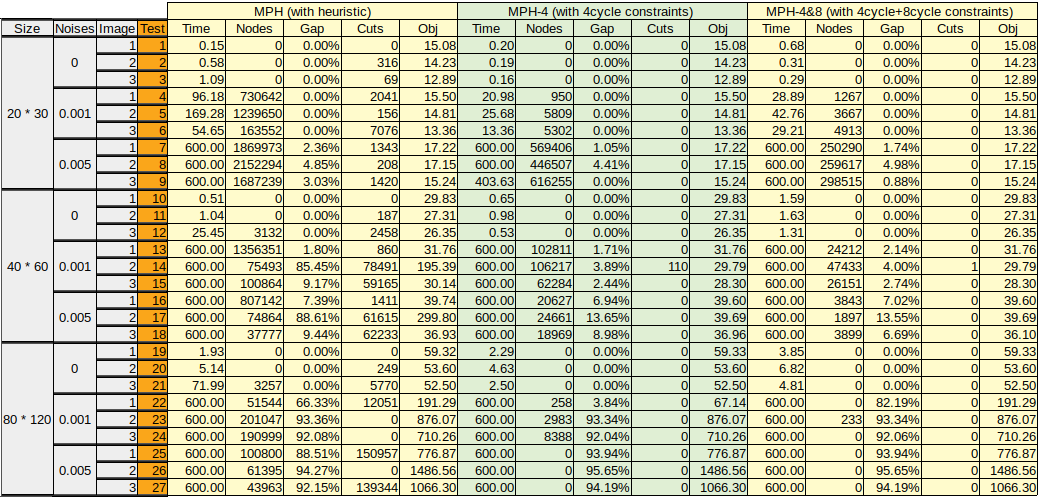}
\caption{Table on MPH, MPH-4 and MPH-4\&8.}
\label{ex1-2}
\end{figure}

\subsubsection{MPH vs MPH-4 and MPH-4\&8}
We compare whether adding few facet-defining multicut constraints as initial constraints to MPH improves computation. We test the performance of adding only $4$-cycle constraints (MPH-4) and adding both $4$-cycle and $8$-cycle (MPH-4\&8). The results are shown in Figure~\ref{ex1-2}. We notice that after adding these cycle constraints, Cplex rarely add any additional cuts to MPH. We also note that in general, adding $4$-cycle constraints helps on improving the performance. For instance, MPH-4 reduces the optimality gap significantly on test $14$, test $17$ and test $22$. In addition, compared to MPH-4, the experiments shows that adding the $8$-cycle constraints seems harmful in most cases.

\subsubsection{Results on segmentation and denoising}

Upon solving our MILP~\eqref{model2}-(\ref{model2}e), the active edges ($x_e=1$) together with the multicut constraints~\eqref{multic} form a valid segmentation, and the fitting variables ($w$) removes noise.
Although only an approximate formulation, the segmentation results of most tests (except for Test $25$-$27$) already achieve ``optimal'' compared to the ground truth. An illustration of the denoising results (as well as segmentation) can be seen in Figure~\ref{ex1-3}, where the first row are the $40\times 60$ images with $0.005$ Gaussian noise, and second row the results from MPH-$4$.

\begin{figure}[h]
\centering
\includegraphics[width=0.32\columnwidth]{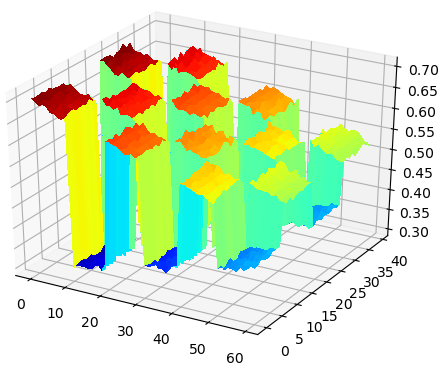}
\includegraphics[width=0.32\columnwidth]{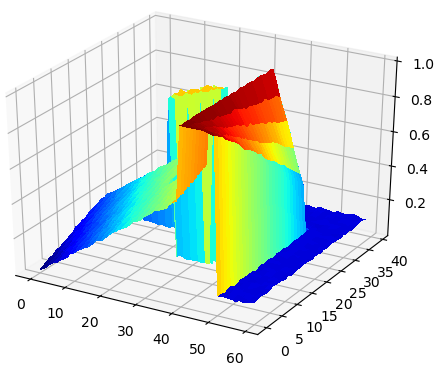}
\includegraphics[width=0.32\columnwidth]{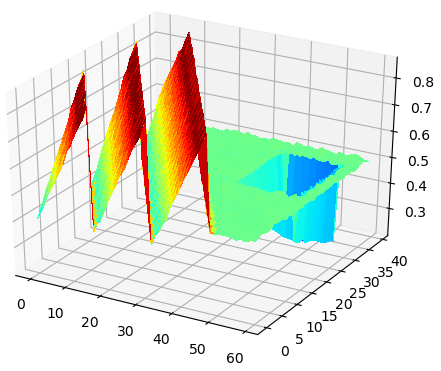}\\
\includegraphics[width=0.32\columnwidth]{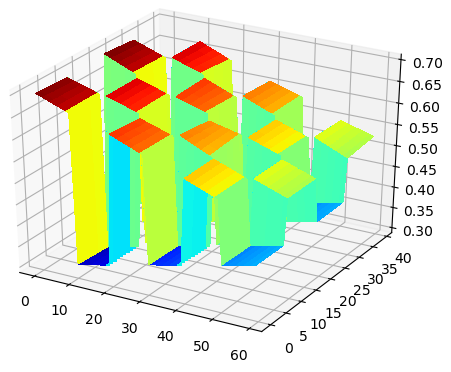}
\includegraphics[width=0.32\columnwidth]{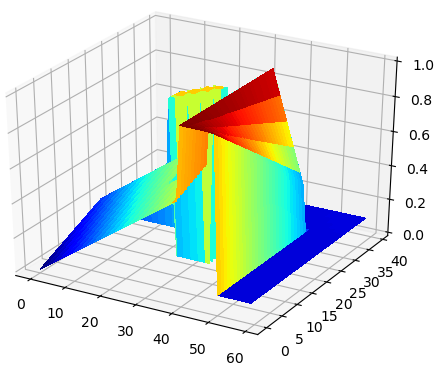}
\includegraphics[width=0.32\columnwidth]{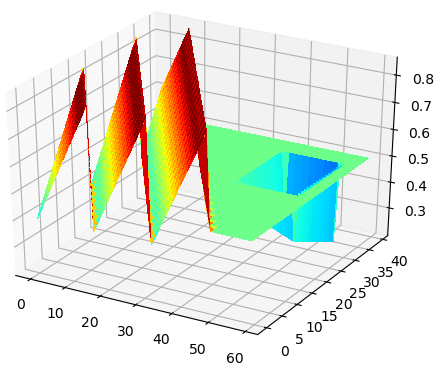}
\caption{Top: images ($40\times 60$) with Gaussian noise. Bottom: results from MPH-4.}
\label{ex1-3}
\end{figure}

%------------------ 	
\subsection{Detailed comparison on real images}
%------------------ 	
We further conduct experiments on two real depth images with $2$ different sizes ($600$ pixels and $2400$ pixels), which are generated from the disparity maps of the Middlebury data set~\cite{4270216} (shown in Figure~\ref{image:real}). 

\begin{figure}[h]
\centering
\includegraphics[width=0.4\columnwidth]{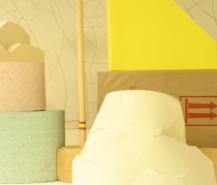}
\includegraphics[width=0.39\columnwidth]{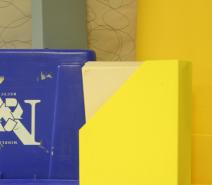}\\
\includegraphics[width=0.4\columnwidth]{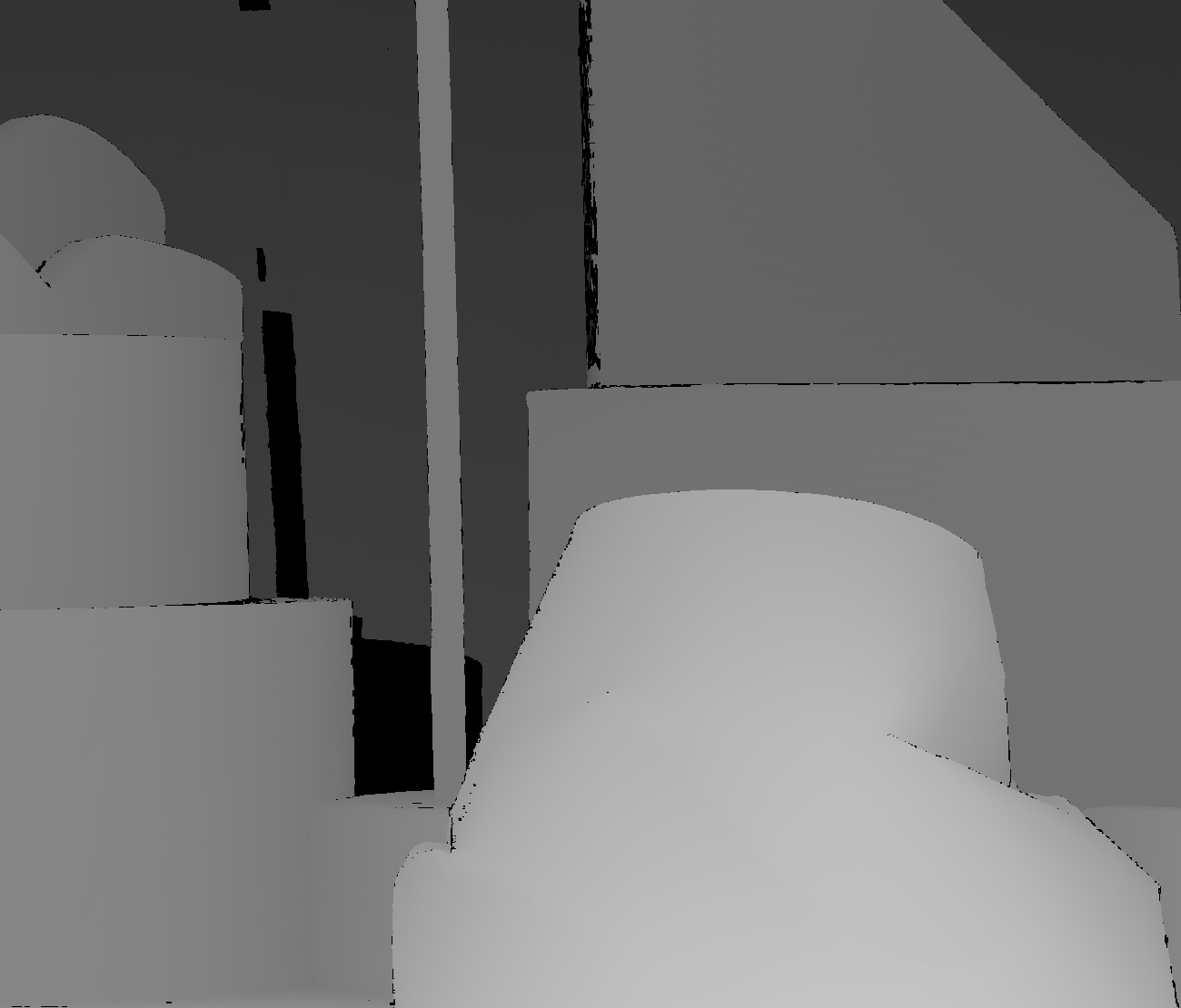}
\includegraphics[width=0.39\columnwidth]{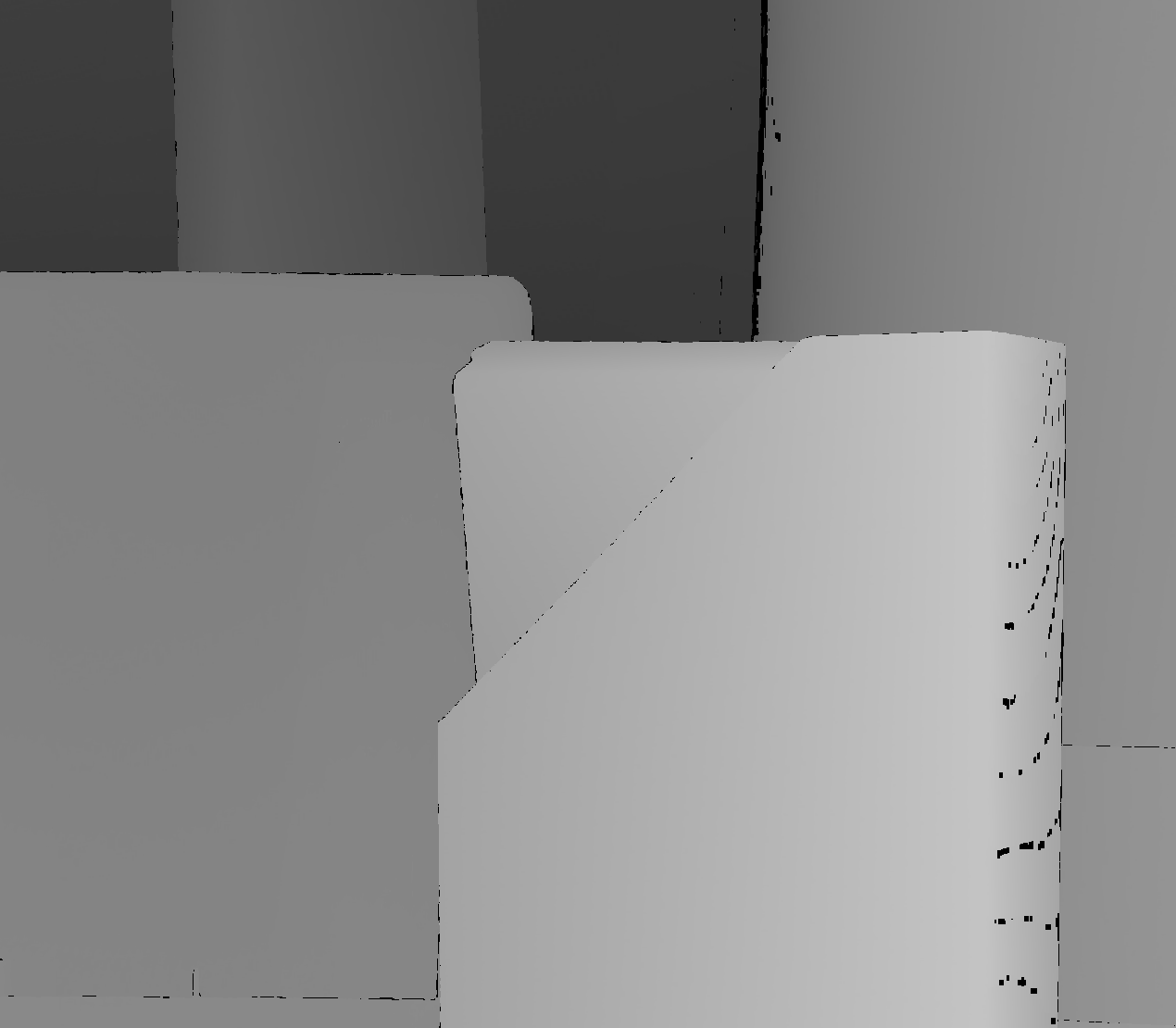}
\caption{Top: Two images fromfrom~\cite{4270216}. Bottom: their disparity maps.}
\label{image:real}
\end{figure}

According to the performance  of the models in previous section, we choose to test different variants (with respect to $\xi$ and time limit) of MPH-4-F (MPH with the $4$-edge cycle multicut constraints using the facet-defining searching strategy). Since real images already contain noise, we do not add extra noise. We also run each experiment $3$ times and only report the medium. All the results are shown in Figure~\ref{table3}.

\begin{figure}[h]
\centering
\includegraphics[width=0.99\columnwidth]{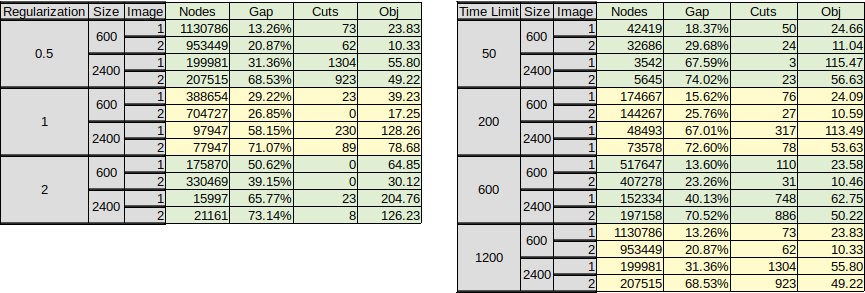}
\caption{Table of tests on MPH-4-F with different regularization parameters $\xi$ and time limits.}
\label{table3}
\end{figure}

\subsubsection{Regularization parameter $\xi$}
The regularization parameter $\xi$ is introduced to penalize the perimeter as well as the number of segments. The larger $\xi$ is, the fewer the segments are.
In this section, we conduct experiments on using $3$ different value of parameter $\xi$ ($0.5$, $1$ and $2$), and the time limit is set to $1200$ seconds. 

The computational results are shown in the left table of Figure~\ref{table3}. However, since the objective functions contain both fitting and regularization terms,  their absolute values is not comparable. Instead, we   visualize the segmentation  results  in Figure~\ref{ex:real}. It is obvious to see that the number of segments decreases as $\xi$ increases.

\begin{figure}[h]
\centering
\includegraphics[width=0.32\columnwidth]{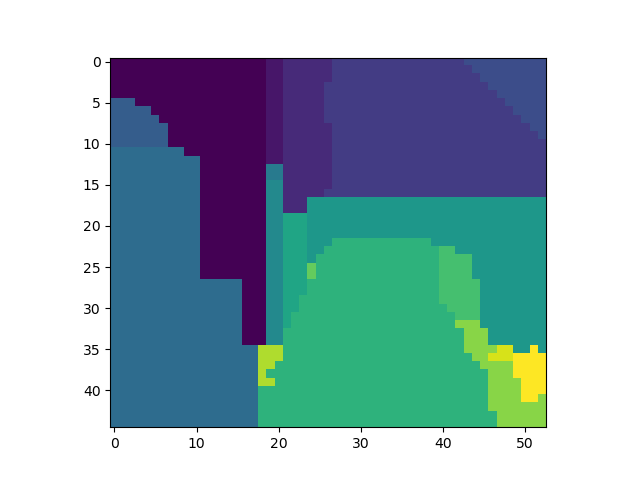}
\includegraphics[width=0.32\columnwidth]{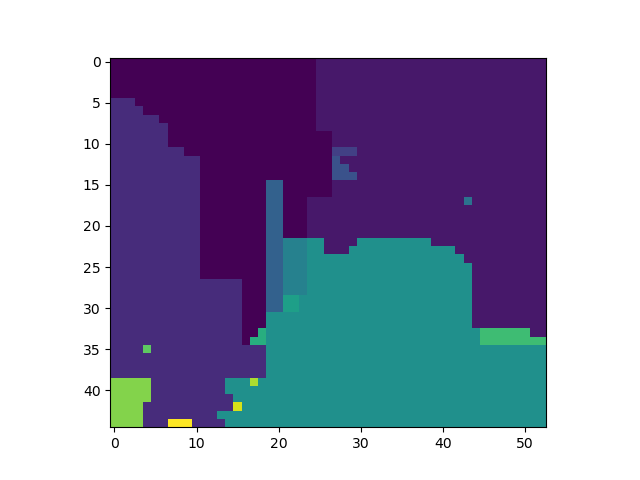}
\includegraphics[width=0.32\columnwidth]{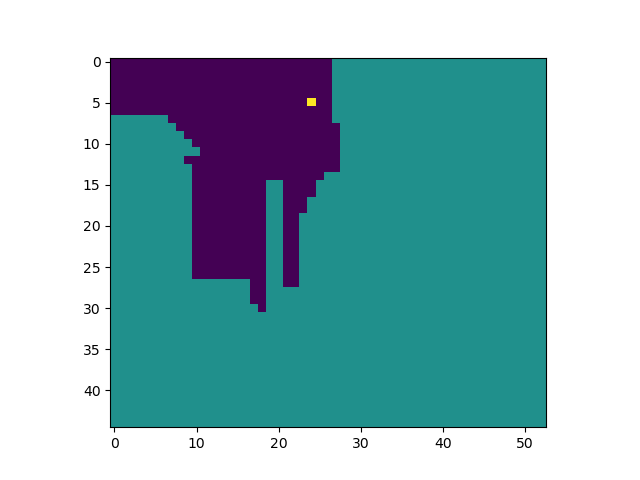}\\
\includegraphics[width=0.32\columnwidth]{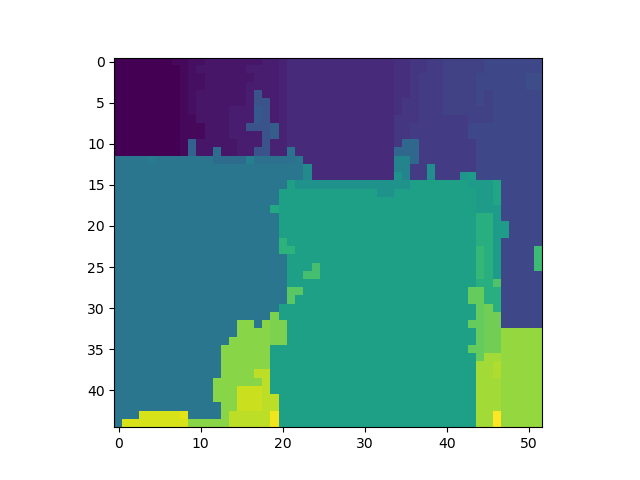}
\includegraphics[width=0.32\columnwidth]{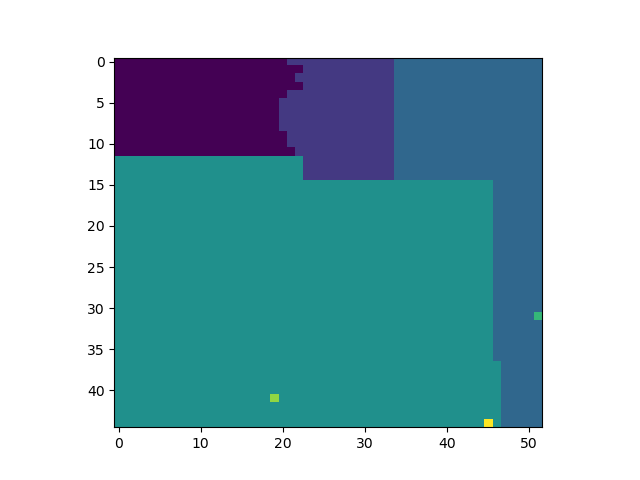}
\includegraphics[width=0.32\columnwidth]{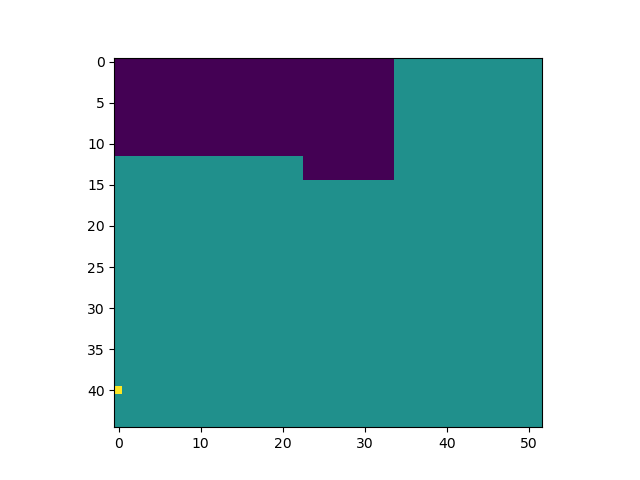}
\caption{Segmentation results as the $\xi$ increases from left to right.}
\label{ex:real}
\end{figure}

\subsubsection{Time limit}

In this section, we conduct experiments on adopting $4$ time limits ($50$, $200$, $600$ and $1200$ seconds), and we set $\xi=0.5$. The computational results are shown in the right table of Figure~\ref{table3}. Since none of tests finds the optimal, the performance could possibly be further improved by extending the time limit. In addition, a shorter time limit is still possible to produce a solution with acceptable gap, especially for images witch smaller size. Figure~\ref{time limit} visualizes the optimality gap with respect to time limit. As can be predicted, when time limit increases, the optimality gap drops.

\begin{figure}[h]
\centering
\includegraphics[width=0.62\columnwidth]{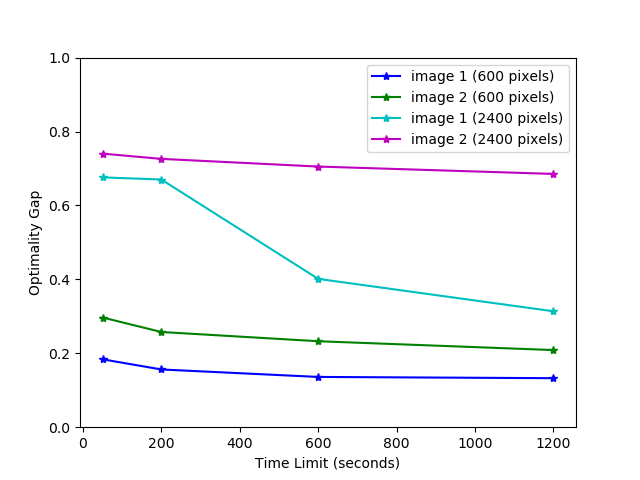}
\caption{Optimality gap decreases as time limit increases.}
\label{time limit}
\end{figure}

%------------------------------------------
\section{Conclusions}
\label{sec:conclusion}
In this paper, we have presented a unsupervised and non-parametric model that  approximates a discontinuous piecewise affine function to fit the given data. We formulate it as a MIP and solve it with a standard optimization solver. Although not an exact model in $2$D, the inclusion of multicut constraints enables a feasible segmentation of the image domain. Thus,  a corresponding piecewise affine function can be easily reconstructed.

The computational complexity is the main bottleneck of our approach. To tackle with it, we add two different sets of facet-defining inequalities to our MIP. We also implemented a special heuristic algorithm that finds a feasible segmentation, which is used as an initial integer solution to the MIP solver.
%We present an unsupervised image segmentation framework for piecewise linear scenarios together with a novel piecewise affine Potts model that solves the problems to optimality. 
We conducted extensive experiments on different variants of our model and study the effects of adjusting model parameters. We demonstrate the feasibility of our approach by its  applications to segmentation and denoising   on both synthetic and real depth images.
%Although GMC is $\mathcal{NP}$-hard, we can still solve large images efficiently by first decomposing them into sub-images, and compute GMC within the sub-images. A new graph is constructed based on the segmentation of all sub-images, and MC is applied to get the final segmentation.

As for future work, the $8$-neighbor relations of the square grid graph in $2$D is worth investigating, as well as its generalization to 3$D$ images. Furthermore, we will extend this work beyond the scope of image segmentation and denoising to deal with other applications, such as signal compression\cite{Duarte2012SignalCI} and optical flow~\cite{8417969}.

%
% ---- Bibliography ----
%
% BibTeX users should specify bibliography style 'splncs04'.
% References will then be sorted and formatted in the correct style.
%
% \bibliographystyle{splncs04}
% \bibliography{mybibliography}
%
%----------------------------------------------------------------
%\bibliographystyle{unsrt}
\bibliographystyle{splncs04}
\bibliography{research}

\end{document}